\documentclass[11pt,A4paper]{amsart}
\usepackage{verbatim, amscd, epsfig,amsmath,amsthm,amstext,amssymb,hyperref,bm
}
\usepackage[all]{xy}

\theoremstyle{plain}
\newtheorem*{theorem*}{Theorem}
\newtheorem{theorem}{Theorem}[section]
\newtheorem{lemma}[theorem]{Lemma}

\theoremstyle{definition}
\newtheorem{definition}[theorem]{Definition}

\newtheorem {proposition} {Proposition}

\newcommand{\R}{\mathbb{ R}}
\newcommand{\C}{\mathbb{ C}}
\newcommand{\Z}{\mathbb{ Z}}

\renewcommand{\H}{\mathbb{ H}}

\renewcommand{\P}{\mathbb{ P}}
\newcommand{\HP}{\H\P}
\newcommand{\CP}{\C\P}

\DeclareMathOperator{\tr}{tr}

\newcommand{\leftsuperscript}[2]{{\vphantom{#2}}^{#1}{#2}}

\newcommand{\residue}{\mathrm{Res}}

\begin{document}
\title
{Harmonic Maps and Integrable Systems}
\author{Emma Carberry}
\date{\today}
\maketitle

\begin {abstract}
This article has two purposes. The first is to give an expository account of the integrable systems approach to harmonic maps from surfaces to Lie groups and symmetric spaces, focusing on spectral curves for harmonic 2-tori. 
The most unwieldy aspect of the spectral curve description is the periodicity conditions and the second aim is to present four different forms for these periodicity conditions and explain their equivalence. 
\end {abstract}

\section {Introduction}
%
%

Harmonic maps from a Riemann surface $\Sigma $ into a Lie group $ G $ or symmetric space $ G/K $ have been widely studied by expressing the harmonic map in terms of a loop of flat connections in a principal $ G$-bundle over $\Sigma$. When the domain is a 2-torus, there has been success, for a number of target spaces, in giving a complete description of the harmonic map in terms of algebraic geometry \cite{Hitchin:90, PS:89, Bobenko:91, McIntosh:95, McIntosh:96, FPPS:92, McIntosh:01, CW:08, Carberry:09}). This has been achieved through the construction of an algebraic curve  $X$, called the \emph{spectral curve, } together with a line bundle on $ X $ and some additional data, from  which it is possible to recover the harmonic map $ f $ up to gauge transformations.  The harmonic map gives a linear flow in a sub-torus of the Jacobian of $ X $ and through this viewpoint we obtain an explicit realisation of the harmonic map equations as an algebraically completely integrable Hamiltonian system.

The spectral curve $ X $ must satisfy crucial \emph{periodicity conditions} which ensure that the corresponding harmonic map is defined on a 2-torus, rather than merely on its universal cover. These periodicity conditions most naturally come in two layers. The first of these corresponds to the double-periodicity of a certain harmonic section of a principal $ G $-bundle over the complex numbers whilst the second is additionally required to recover an actual harmonic map. In Theorems ~\ref {theorem:periodicity} and \ref{theorem:singular} we give a number of different formulations of both layers of these periodicity conditions together with a simple proof of their equivalence, in the process explaining why the periodicity criteria given in \cite {Hitchin:90} and \cite {McIntosh:01} are in fact the same.  In particular this makes clear that periodicity depends only upon the spectral curve and its projection to the projective line, not on the remaining spectral data.

The periodicity conditions place transcendental restrictions on the algebraic description of a harmonic 2-torus provided by the spectral curve data.  Given a positive integer $ g $, it natural to ask whether there exist spectral curves of (arithmetic) genus $ g $ satisfying these periodicity conditions. This is geometrically significant because an immediate application of the spectral curve description is that it shows that a harmonic map of a 2-torus may be deformed through an $ n $-dimensional family, where $ n $  is the dimension of the sub-torus of the Jacobian of the spectral curve alluded to above and is determined by the spectral genus. Furthermore for several target spaces fundamental invariants of the harmonic maps such as energy and geometric complexity have been bounded below in terms of the spectral genus \cite {BLPP:07, Haskins:04}.
The existence of harmonic tori of arbitrary spectral genus $ g $\footnote {Excepting spectral genus $ g = 1 $ for constant mean curvature tori in $\R ^ 3 $ for which no solutions exist.} has been proven for a number of target spaces \cite {EKT:93, Jaggy:94, Carberry:04, CM:03, CS:12}, by demonstrating the existence of spectral curves of the appropriate genus satisfying periodicity conditions. It is a measure of how unwieldy these conditions are that in all except the last case this was shown only for spectral curves in some neighbourhood of a curve of geometric genus zero, in order to simplify the computations. (In\cite {CS:12} it was shown that spectral curves satisfying the periodicity conditions are in fact dense in the space of all spectral curves for the case of constant mean curvature tori in $ S ^ 3 $.) The periodicity conditions are by far the most difficult part of the spectral curve data to handle, and it is hoped that giving various forms of these conditions and elementary proofs of the relationships between them will render these  conditions more transparent.

More broadly, spectral curves are an important tool in integrable systems.  They provide an algebro-geometric description of the solution to a family of differential equations in Lax form $ dA_\lambda (t) = [A_\lambda (t), B_\lambda (t)]$, hence introducing a new set of techniques to the study of these differential equations 
 and providing an explicit realisation of such an equation as a completely integrable Hamiltonian system.   Moreover as Griffiths has shown \cite {Griffiths:85}, whenever a natural cohomological condition on the Lax pair is satisfied, the spectral curve approach provides an explicit linearisation of the Lax equation as the equation can be reformulated as a linear flow in the Jacobian of the spectral curve. Familiar integrable systems such as the Toda lattice, geodesics on an ellipsoid, 
the Euler equations for a free body moving around a fixed point and Nahm's equations are accounted for in this way  \cite {AM:80:1, AM:80:2, Hitchin:84}, and in each case the spectral curve is simply the characteristic polynomial  solution to the Lax pair. However the harmonic map story involves some additional subtleties. In particular, the harmonic map corresponds not to a single solution to a Lax pair equation as above but rather an entire algebra of them, prompting a more sophisticated spectral curve construction than simply taking the characteristic polynomial.
 We explain the traditional constructions which use either holonomy of a family of flat connections or  polynomial killing fields, which are the solutions to the Lax pair mentioned above.
We also  describe a more general approach, for which the existence of both the family a flat connections and the Lax pair is superfluous, and discuss the merits of each construction.

For ease of exposition for the equivalent characterisations of periodicity we restrict ourselves to the specific case of harmonic maps from a 2-torus to $ S ^ 3 $, allowing for a more explicit treatment. This was the case studied in detail by Hitchin in \cite {Hitchin:90}. The various manifestations of the periodicity conditions are tailored to this case but these are simply different ways of expressing an element of the \v{C}ech 
cohomology $ H ^ 1 (X,\mathcal O) $ for an algebraic curve $ X $, and it is elementary to give the appropriate modification for spectral curves for another class of harmonic maps.

In section ~\ref{section:content} we give basic background material and set up the standard integrable systems description of harmonic maps by a family of flat connections. Section ~\ref{section:spectral} contains an analysis of the periodicity conditions, relating their various forms. The final section describes an alternative approach to spectral curves of harmonic maps using polynomial Killing fields as well as the more general approach of the multiplier curve, for which no family of flat connections is needed.

It is my pleasure to express my gratitude to Hyam Rubinstein for his support and wisdom over many years and to the Hyamfest organisers for the opportunity to contribute to these proceedings in his honour.

\newcommand{\g}{\mathfrak{g}}
 
 \section {Harmonic maps of surfaces into Lie groups}\label{section:content}

Harmonic maps of surfaces have many applications in geometry. The most famous examples are minimal surfaces, a conformal immersion is harmonic precisely when its image is a minimal (immersed) surface. Physically, minimal surfaces are modelled by soap films or other thin membranes and locally provide the surface of least area with a given boundary. If instead we ask the soap film to enclose a volume, that is we consider soap bubbles, then mathematically we are considering constant mean curvature surfaces, with mean curvature $ H\neq 0 $. The mean curvature is proportional to the pressure difference between the inside and outside of the bubble. Surfaces of constant mean curvature are characterised by having a harmonic Gauss map and so are also a geometric manifestation of harmonic maps. Other geometrically interesting surfaces described by harmonic maps include Willmore surfaces, which are characterised by the harmonicity of their conformal Gauss map.

In this section we introduce the approach to harmonic maps of surfaces using integrable systems and spectral curves, and provide the relevant background information. 

\subsection {Introduction to harmonic maps}

Recall that a smooth function $ f: U\subset\R ^ m\rightarrow\R $ is {\em harmonic} if its Laplacian vanishes:
\[
\Delta f =\nabla\cdot \nabla f =\dfrac {\partial ^ 2F} {\partial (x^1) ^ 2} +\cdots +\dfrac {\partial ^ 2f} {\partial (x ^ m) ^ 2} = 0.
\]
We define the {\em energy} of $ f $ on a compact set  $ U $ to be
\[
E (f) =\frac 12\int _U\left|df\right| ^ 2\,d x.
\]
Then for any smooth $ u\colon U\to \mathbb{R} $ vanishing on the boundary of U,
\[
    \frac{d}{d\varepsilon}\Big|_{\varepsilon = 0} E(f+\varepsilon u) = \int_U \nabla f \cdot \nabla u \, d x = -\int_U u \Delta f\,d x
\]
where the last equality comes from applying the divergence theorem
to $ F = u\nabla f $. Harmonic functions are thus characterised by being critical for the energy functional.

Harmonic maps between pseudo-Riemannian manifolds are a generalisation of harmonic functions and as such satisfy a generalisation of the Laplace equation. Given a smooth map $ f: M\rightarrow N $ where $ (M, g) $ and $ (N, h) $ are pseudo-Riemannian manifolds, the  {\em second fundamental form} of $ f $ is $\nabla (df) $, 
where $\nabla $ is the connection on $\mathrm{Hom} (TM, TN) $ induced from the Levi-Civita connections on $ M $ and $ N $, namely
\[
(\nabla (df)) (X, Y) =\nabla ^ {TN}_{df (X)} (df (Y)) - df (\nabla ^ {TM}_X (Y)).
\]
The  {\em tension} of $ f $ is defined to be the trace of its second fundamental form with respect to $ g $, that is
\[
\tau (f) = g ^ {ij}\nabla (df) \left (\frac {\partial} {\partial x ^ i},\frac {\partial} {\partial x ^ j}\right)
\]
where
\[
g ^ {ij} = g (dx ^ i, dx ^ j) 
\]
and $ x ^ i $ is a local coordinate system for $ M $.
\begin {definition}
We say that a smooth map $ f: M\rightarrow N $ between pseudo-Riemannian manifolds is {\em harmonic} if the tension $\tau (f)=\tr_g (\nabla (df)) $ of $ f $ vanishes.\end {definition}
Writing $ x ^ i $ for local coordinates on $ M $ and $ y ^\alpha $ for local coordinates on $ N $,  the generalisation of Laplace's equation to Riemannian manifolds is 
\[
\tau (df) = g ^ {ij}\left (\frac {\partial ^ 2f ^\gamma} {\partial x ^ i\partial x ^ j} +\frac {\partial f ^ \beta} {\partial x ^ i}\frac{\partial f ^\alpha} {\partial x ^ j}\leftsuperscript {N} {\Gamma}^\gamma_{\alpha\beta} -\leftsuperscript {M} {\Gamma ^ k_{ij}}\frac {\partial f ^\gamma} {\partial x ^ k}\right)\frac {\partial} {\partial y ^\gamma}= 0.
\]
For $ f:\R ^ m\rightarrow\R $,
\begin {align*}
(\nabla (df)) (\dfrac {\partial} {\partial x ^ i}, \dfrac {\partial} {\partial x ^ j})
& = \nabla ^ {T\R}_{\frac {\partial y} {\partial x ^ i}\frac{\partial}{\partial y}}\left (\dfrac {\partial y} {\partial x ^ j}\dfrac {\partial} {\partial y}\right) -df \Bigl(\nabla ^ {T\R ^ m}_{\frac {\partial} {\partial x ^ i}} \Bigl(\dfrac {\partial} {\partial x ^ j}\Bigr)
& =\dfrac {\partial ^ 2y} {\partial x ^ i\partial x ^ j}\dfrac {\partial} {\partial y}
\end {align*}
so that we recover the standard Laplace equation.
%

Another fundamental example is obtained by setting the domain to $\R $
as $ f:\R\rightarrow N $ is harmonic precisely when
\[
(\nabla (df)) (\dfrac {\partial} {\partial t}, \dfrac {\partial} {\partial t}) =\nabla ^ {TN}_{df (\frac {\partial} {\partial t})} (df (\dfrac {\partial} {\partial t})) - df (\nabla ^ {T\R}_\frac {\partial} {\partial t} (\dfrac {\partial} {\partial t}))
\]
vanishes, that is when $ f $ is a geodesic.


As with harmonic functions, harmonic maps are critical for the energy functional, and a straightforward computation yields 
\begin {proposition}
A smooth map $ f: M\rightarrow N $ between Riemannian manifolds is harmonic if and only it is critical for the {\em energy functional}
\[
E (f) =\frac 12\int_{U}\left\|df\right\| ^ 2d \mathrm{vol}_M,
\]
in the sense that for any one-parameter variation $ f_t $ of $ f $,
\[
\left.\frac {d} {dt}\right|_{t = 0} E (f_t) = 0.
\]
\end {proposition}

The standard physical interpretation of this result is that one should think of the target manifold $ N $ as being made of marble and the domain manifold $ M $ as formed of rubber, the harmonic map stretches the domain into a shape which locally minimises energy.

We shall consider now harmonic maps of a Riemann surface into a Lie group with a bi-invariant pseudo-metric and  both assumptions simplify the harmonic map equation. The condition that the group possess a bi-invariant pseudo-metric is not onerous, holding true for any reductive Lie group, that is one whose Lie algebra can be written as a direct sum of a semisimple Lie algebra and an abelian one. For such groups we can combine the Cartan-Killing form on the semisimple Lie algebra with any form on the abelian part. 
We shall be interested in harmonic maps of surfaces into Lie groups and  symmetric spaces. The latter provide more geometric applications but as we shall see, using the totally geodesic Cartan embedding, one can study harmonic maps into symmetric spaces in terms of harmonic maps into Lie groups and vice versa. 

\subsection {Harmonic maps from a surface to a Lie group}

Any smooth map into a Lie group $ G $ is locally described by a zero curvature condition. To see this we introduce the Maurer-Cartan form on $ G $, namely the $\mathfrak{g}$-valued 1-form on $G$ that is
left-invariant and acts as the identity on $T_e G = \mathfrak{g} $. It is given explicitly by
\[
\omega(v) = (L_{g^{-1}})_* v,\quad v\in T_gG,
\] 
where $ L_g $ is left multiplication by $ g $.
For a linear Lie group, writing $ g: G\rightarrow M_{n\times n} (\R) $ then
\[
\omega_a = g (a)^ {- 1} dg_a
\]
which is often written as
\[
\omega = g^ {- 1} dg.
\]
It is customary to use this notation even when $ G $ is not linear.

A straightforward computation shows that 
for any Lie group $ G $, the Maurer-Cartan form $\omega$
satisfies the {\it Maurer-Cartan equation}
\[
d\omega + \frac {1}{2}[\omega \wedge\omega]=0.
\]
Here for $\mathfrak g $-valued 1-forms $\varphi $ and $\psi $ we define
\[
[\varphi\wedge\psi ] (X, Y) = [\varphi (X),\psi (Y)] - [\varphi (Y),\psi (X)].
\]
The Maurer-Cartan equation can be viewed as a zero-curvature condition, on a local coordinate neighbourhood we can define a connection by $ d +\omega $, and then the Maurer-Cartan equation precisely says that this connection is flat.
That such an equation locally describes smooth maps into $ G $ now follows from the following observation of Cartan.
\begin{lemma}[Cartan]
Given any smooth map $f$ from a  manifold $ M $ to
$G$, we may pull back the Maurer-Cartan form to obtain a
$\g$-valued 1-form $\phi = f ^*\omega $ on $ M $ satisfying
\[
d\phi+\frac 12 [\phi\wedge\phi]=0.
\]
Two maps define the same form on $M$ if and only if they differ by
left translation by a fixed element of $G$. 

Conversely, given a
$\g$-valued 1-form $\phi$ on a simply connected manifold $ M $ satisfying the Maurer-Cartan
equation, we may integrate it to obtain a smooth map $f: M\to G$ such
that $f^*(\omega)=\phi$, where $f$ is defined only up to left
translation.
 For $ G $ linear, $\phi = f^ {- 1} df $.
\end{lemma}

 We can express the condition for $ f: M\rightarrow G $ to be harmonic in terms of $\phi = f ^*\omega $. 
Trivialising $ TG $ by left-translation, we shall write
\begin {itemize}
\item $\nabla ^ L $ for the left connection (in which the left-invariant vector fields are parallel).
\item $\nabla ^ R $ for the right connection
\item $\nabla $ for the Levi-Civita connection with respect to the Killing form.
\end {itemize}
These connections are related by
\[
\nabla ^ R =\nabla^ L +\omega\text { and }\nabla =\frac 12 (\nabla ^ L +\nabla ^ R) =\nabla ^ L +\frac 12\omega.
\]
Using this, a computation shows that the tension of a map $ f: M\rightarrow G $ is given by
\[
\tau (f) = d ^*\phi
\]
where $ d ^*$ is the adjoint of $ d $ and as above, $\phi = f ^*\omega $.
Hence for $ M $ a Riemann surface and $ G $ a Lie group with bi-invariant pseudo-metric, $ f: M\rightarrow G $ is harmonic if and only if
\[
d*\phi=0,
\]
where $*$ is the Hodge star operator
\[
*dx = dy,\quad*dz = - dx.
\]
Thus a harmonic $ f:U\subset M\rightarrow G $ on a simply connected subset $ U $ of a Riemann surface $ M $ is equivalent to a $\g $-valued 1-form $\phi $ ($\phi = f^ *\omega $) satisfying
\[
d\phi+\frac 12 [\phi\wedge\phi]=0,\quad d*\phi=0.
\]
That is, writing $\phi =\phi_zdz +\phi_{\bar z} d\bar z $ we have
\[
\dfrac {\partial\phi_z} {\partial\bar z} -\dfrac {\partial\phi_{\bar z}} {\partial z} + [\phi_{\bar z},\phi_z] = 0
\] and
\[
\dfrac {\partial\phi_z} {\partial\bar z} +\dfrac {\phi_{\bar z}} {\partial z} = 0
\]
or more symmetrically,
\begin {equation}\dfrac {\partial\phi_z} {\partial\bar z} +\frac 12 [\phi_{\bar z},\phi_z]  = 0\label {eq:harmonic1}\end {equation}
and
\begin {equation}\dfrac {\partial\phi_{\bar z}} {\partial z} +\frac 12 [\phi_z,\phi_{\bar z}]  = 0.\label {eq:harmonic2}\end {equation}
The key observation which underlies the integrable systems approach to such harmonic maps is the fact that these equations can be expressed as the requirement that a certain family of connections has zero curvature \cite {Pohlmeyer:76, Uhlenbeck:89}. 
Namely, for $\lambda\in\C ^\times $, consider the connections 
\begin {equation}
\nabla_\lambda =
 \nabla ^ L +\phi_\lambda \label{eq:connections}
\end {equation}
 where
\[
\phi_\lambda = \frac 12 (1+ \lambda ^{-1})\phi_z dz +\frac 12 (1 - \lambda)\phi_{\bar z}d\bar z.
\]
Then it is easy to check that \eqref {eq:harmonic1} and \eqref {eq:harmonic2} are equivalent to the each of the connections $ \nabla_\lambda $, $\lambda\in\C ^\times $ being flat, that is satisfying
\begin{equation}\label {eq:family}
d\phi_\lambda +\frac 12 [\phi_\lambda\wedge\phi_\lambda ]= 0.
\end {equation}
The zero-curvature formulation enables us to give some  solutions to the harmonic map equations simply by integrating a pair of commuting vector fields on a finite dimensional space, that is by solving a pair of ordinary differential equations. This is far simpler than dealing with the harmonic map equations directly and maps obtained in this way are said to be of finite-type.
This formulation is also the basis for describing harmonic maps of finite type in terms of spectral curves. The zero-curvature equations become a linear flow in the Jacobian of the spectral curve. This has led to much interesting moduli-space information as well as descriptions of important invariants such as energy. 
%

\newcommand{\Ad}{\mathrm{Ad}}

\section {Harmonic tori in $ S ^ 3 $: spectral curve data and periodicity}
\label {section:spectral}

We begin by describing Hitchin's holonomy construction of a spectral curve for harmonic or constant mean curvature immersions from a 2-torus to $ S ^ 3 $ \cite {Hitchin:90}. In this context we then provide a number of different formulations of both layers of the periodicity conditions and an explicit proof of their equivalence. Although for concreteness we have focused on harmonic maps to a particular target, as the reader will see it is straightforward to extend these interpretations to a general algebraic curve. 

We give, in Theorem~\ref {theorem:periodicity}, four different ways of describing an element of $ H ^ 1 (X,\mathcal O) $, for $ X $ a spectral curve. Namely, we describe this vector using the periods of a differential $\Theta $ of the second kind, the principal parts of this differential, derivatives of Abel-Jacobi maps with base points the singularities of $ \Theta $, and the derivative of a family of line bundles defined using its principal parts.

Constant mean curvature surfaces in $\R ^ 3 $ are characterised by having a harmonic Gauss map, which of course has target $ S ^ 2\subset S ^ 3 $. These are thus included in description below, once we take account of an additional condition needed to ensure that the constant mean curvature surface is doubly periodic, not merely its Gauss map. Analogously spectral data for constant mean curvature tori in the 3-sphere or in hyperbolic 3-space are quite similar to that described below \cite {Bobenko:91}.

Hitchin \cite{Hitchin:90}  studied harmonic maps $f: T ^ 2\rightarrow SU (2)$ by considering the holonomy  of a family of flat connections $\nabla_\lambda$. Writing $H ^\gamma_\lambda (z)$ for the holonomy of $\nabla_\lambda$ around $\gamma\in\Pi_1 (T ^ 2, Z)$, we may define an algebraic curve $X$, the \emph{spectral curve}, by taking the eigenline curve of $H ^\gamma_\lambda (z)$ as described below. 
Since the fundamental group of $T ^ 2$ is abelian this definition is independent of the choice of $\gamma$, and as choosing a different $z\in T ^ 2$ changes the holonomy only by conjugation, it is independent of the choice of $z$. However the eigenline themselves do depend on $z$, and we obtain in this way an eigenline bundle $\mathcal E_ z$ for each $ z\in T ^ 2 $, and a linear map $T ^ 2\rightarrow\mbox {Jac} (X)$.

As explained above, given a harmonic map $ f: T ^ 2\rightarrow S ^ 3\cong SU (2) $ and writing $\phi = f^ {- 1} df $, the $ S ^ 1 $-family of one-forms 
\[
\varphi _\lambda = \phi_\lambda = \frac 12 (1+ \lambda ^{-1})\phi' +\frac 12 (1 - \lambda)\phi''
\]
each satisfies the Maurer-Cartan equation. Allowing $\lambda $ to be any non-zero complex number, we obtain a $\C ^\times $-family of flat connections 
\[
\nabla_\lambda =\nabla ^ L +\phi_\lambda 
\]
 in the bundle 
 $ V = f ^*(T SU (2)\otimes \C) $, where $\nabla ^ L $ is the connection induced by trivialising the tangent bundle by left translation. For $\lambda\in S ^ 1 $, the holonomy of these connections is valued in $ SU (2) $ and for $\lambda\in\C ^\times $ it is valued in the corresponding complex group $ SL (2,\C) $. Since $ SU (2) $ are the unit quaternions, the $ SU (2) $-structure in the bundle $ V $ may be exhibited by a quaternionic structure $ j $ (that is an anti-linear involution whose square is $ -1 $) together with a symplectic form $\omega $.
 Note that we could equivalently consider these connections in the $ SU (2) $-principal bundle $ P $ given by the pullback under $ f $ of the trivial bundle $ SU (2)\times SU (2)\rightarrow SU (2) $ with projection $ ( g_1, g_2)\mapsto g_1 $ and action $ ( g_1, g_2)\cdot h = (g_1, g_2h) $, and then $ f ^*(TSU (2)\otimes\C) = \Ad P $. 
When the connections $\nabla_\lambda $ arise from a harmonic map $ f: T ^ 2\rightarrow S ^ 3 $, both $\nabla_1 $ and $\nabla_{-1} $ are not merely flat but trivial, since they represent the left and right connections in the pullback of the tangent bundle. The harmonic map $ f $ is precisely the gauge transformation  between these two trivialisations. 

If we simply have a family of flat connections $\nabla_\lambda $ of the form given above in a rank two complex vector bundle $ V $ with $ SU (2) $ structure, where $\nabla ^ L $ is a given connection in this bundle, then this data does not quite correspond to a harmonic map to $ SU (2) $, instead the geometric interpretation of the corresponding gauge transformation is that it yields a harmonic section $ f $  of  the  bundle $ P $. 

Let $ \pi_1 (T ^ 2, z) $ denote the fundamental group of the torus $ T ^ 2 $ with base point $ z $, and then for each $\lambda\in\C ^\times $ the holonomy $ H_\lambda (z):\pi_1 (T ^ 2, z)\rightarrow SL (2,\C) $ of the connection $\nabla_\lambda $ gives a family of commuting matrices, which hence have common eigenspaces. The fact that this family is abelian is crucial, and from this point of view is the reason why attention is restricted to harmonic maps of genus one domains as this is the only case in which we have a nontrivial yet abelian fundamental group. Unless the holonomy is trivial, then away from isolated $\lambda\in\C ^\times $ it will have a pair of distinct eigenlines. Trivial holonomy corresponds to conformal harmonic map into a totally geodesic $ S ^ 2\subset S ^ 3 $ and is excluded from our considerations. A key result is that there are only finitely many values $\lambda\in\C ^ \times $ at which the holonomy does not have two distinct eigenlines. The proof given by Hitchin \cite [Proposition 2.3] {Hitchin:90}  uses both the fact that an elliptic operator on the compact surface $ T ^ 2 $ can have but a finite-dimensional kernel and that due to the simple structure of the group $ SU (2) $, if $ H_\lambda $ for $\lambda\in S ^ 1 $ leaves fixed a single vector it is necessarily trivial.

The spectral curve of a harmonic map $ f: T ^ 2\rightarrow S ^ 3 $ is the eigenline curve of the holonomy $ H_\lambda $. For $\lambda\in\C ^\times $ such that $ H_\lambda (z) $ has distinct eigenlines, let us denote these by $ E ^ 1_\lambda (z) $ and $E ^ 2_\lambda (z) $. At the isolated points for which these eigenlines are not distinct, we may for each $ z\in T ^ 2 $ define the one-dimensional subspaces $ E ^ 1_\lambda (z), E ^ 2_\lambda (z)\subset V_z $ by analytic continuation. Changing the choice of $ z\in T ^ 2 $ changes the eigenlines by conjugation by a $ SL (2,\C) $ matrix and hence does not affect this branching behaviour. The spectral curve is given by an equation of the form $ y ^ 2 =\tilde a (\lambda) $, for $\lambda \in \C ^\times $ the polynomial $\tilde a $ has a zero of order $ n $ at $\lambda $ precisely when  the eigenlines $ E ^ 1_\lambda (z) $ and $E ^ 2_\lambda (z) $
agree to order $ n $ as measured by the symplectic form $\omega $, that is when $\omega ( E ^ 1_\lambda (z),E ^ 2_\lambda (z)) $ vanishes to order $ n $. For details, see \cite {Hitchin:90}.

The holonomy satisfies $\overline { H}_{\bar\lambda^ {- 1}} ^ t = H_\lambda ^ {- 1} $ so the branching behaviour at $\bar\lambda^ {- 1} $ is identical to that at $\lambda $. To determine the appropriate completion of the open curve one must compute the
 limiting behaviour of the  eigenlines as $\lambda $ approaches $ 0 $. When $ f $ is non-conformal, the eigenlines $ E ^ j_\lambda (z) $ have distinct limits as $\lambda\rightarrow 0 $ and when $ f $ is conformal, they agree in the limit to first order \cite [Propositions 3.5, 3.9, 3.10] {Hitchin:90}. Then letting $ a (\lambda) $ be a polynomial with zeros in $\C ^\times $ described above, the spectral curve $ X $ of a non-conformal harmonic $ f: T ^ 2\rightarrow S ^ 3 $ is the hyperelliptic curve
$ y  ^ 2 = a (\lambda)$ whereas when $ f $ is conformal, it is the curve $ y ^ 2 =\lambda a (\lambda) $.


\newcommand{\Jac}{\mathrm{Jac}}
For each $ z\in T ^ 2 $ the eigenlines of the holonomy $ H_\lambda (z) $ with base point $ z $ then define a holomorphic line bundle $ \mathcal E_z $ on the spectral curve $ X $. These line bundles each have degree $ g + 1 $, where $ g $ denotes the arithmetic genus of the spectral curve $ X $. Fixing a point $ 0\in T ^ 2 $, the resulting map
\begin {align*}
T ^ 2 &\rightarrow\Jac (X)\\
z &\mapsto\mathcal  E _z \otimes \mathcal E _0 ^*
\end {align*}
is linear.

Choosing generators $[ 0, 1 ]$, $[ 0,\tau ]$ for the fundamental group and writing $\mu $, $\nu$ for the eigenvalue functions of the holonomy with respect to thsee generators, define two differentials of the second kind by $\Theta = d\log \mu$, $\Psi = d\log\nu$. The construction can be reversed to yield a harmonic map from spectral data; the following statement is taken from theorem 8.1 and 8.20 of \cite{Hitchin:90}.

\begin{theorem} [\cite {Hitchin:90}]
\label{theorem:Hitchin}
A harmonic map $f: T ^ 2\rightarrow S ^ 3$ uniquely determines
a quadruple $(X,\lambda,\mathcal E_0,\Theta,\Psi)$, satisfying the following:
\begin{enumerate}
\item\label{item:first}$X$ is a hyperelliptic curve $y ^ 2 = a (\lambda)$, with a fixed point free real structure $\rho$ covering involution in the unit circle $\lambda\mapsto\bar\lambda^ {- 1}$ and such that $X$ is smooth at $\lambda^{-1}\{0,\infty\}$;
\item The differentials satisfy $\sigma ^*\Theta = -\Theta $, $\sigma ^*\Psi= -\Psi$ and $\overline {\rho ^*\Theta} = -\Theta $, $\overline {\rho ^*\Psi} =\Psi $ where $\sigma$ is the hyperelliptic involution;
\item 
$\Theta $, $\Psi $ have double poles at  $\lambda^ {- 1}\{0,\infty\}$ and no residues, and are otherwise holomorphic;
\item If the domain of the harmonic map $ f $ is $ T ^ 2 =\R ^ 2/\Lambda $ with $\Lambda = 1\cdot\Z +\tau\cdot\Z $ then $\tau =\frac {\text {p.p.}_{p_0}\Psi} {\text {p.p.}_{p_0}\Theta} $ where p.p. denotes principal part;
\item\label{item:last}$\mathcal E_0\in\mbox {Pic}_{g +1} X$, where $g$ is the arithmetic genus of $X$, is quaternionic with respect to $\rho\sigma$;
\end{enumerate}
together with the {\it periodicity conditions}:
\begin{enumerate}
\item[(P1)] The periods of $\Theta $, $\Psi $ lie in $2\pi\sqrt {-1}\Z$;
\item[(P2)] 
If $\gamma_1 $ is a curve in $ X $ with endpoints the two points in $\lambda^ {- 1} (1) $ and $\gamma_{-1} $ a curve with endpoints the two points in $\lambda ^{-1} (-1) $ then
the integrals of $\Theta $, $\Psi $ over $\gamma_1 $ and $\gamma_{-1} $ are valued in $ 2\pi\sqrt {-1}\Z $.
\end{enumerate}

These periodicity conditions may equivalently be expressed as
\begin {enumerate}
\item [(P1)$'$] There exist meromorphic functions $\mu $ and $\nu $ on $ X\setminus\lambda ^{-1}\{0,\infty \} $ such that $\log\mu $ and $\log\nu $ extend to meromorphic functions on $ X $ satisfying $\Theta = d\log\mu$ and $\Psi= d\log\nu$. These functions have symmetries  $\mu\sigma ^*\mu = \nu\sigma ^*\nu = 1$;
\item [(P2)$'$] We may choose $\mu $, $\nu $ so that $\mu (p) =\nu (p) = 1$ for all $p \in\lambda^ {- 1} \{\pm 1\} $.
\end {enumerate}

Conversely,  $ (X,\mathcal E_0,\mu,\nu) $ as above  determines a harmonic map $ f: T ^ 2\rightarrow S ^ 3 $  (uniquely determined up to the action of $SO (4)$ on $S ^ 3$). 

The harmonic map $f$ is conformal if and only if $0$ and $\infty$ are branch points of the map $\lambda: X\rightarrow\P ^ 1$.
\end{theorem}


%
The first periodicity condition is sufficient to yield a family of flat connections of the form  \eqref {eq:connections} or equivalently a harmonic section of an $SU (2)$ principal bundle over $T ^ 2$. To trivialise this bundle and hence obtain a harmonic map, the second condition is required. If we have $\Theta $, $\Psi $ satisfying both periodicity conditions, the meromorphic functions $\mu $, $\nu $ are each determined only up to sign, which corresponds to having a harmonic map on the torus $\R ^ 2/(2\Lambda) $ rather than on $\R ^ 2/\Lambda $. The next theorem will enable us to variously phrase periodicity condition (P1) in terms of Abel-Jacobi maps on the spectral curve $ X $, principal parts of the differentials $\Theta $, $\Psi $ or in terms of derivatives of the eigenline bundles. 

We first describe how  the above spectral data naturally gives rise to a linear flow in the Jacobian of the spectral curve.
Locally the differential forms $\Theta $ and $\Psi $ may be expressed as the differentials of functions $\int\Theta $ and $\int\Psi $, although these functions are only defined locally their principal parts give well-defined global sections of the 
 sheaf $\mathcal P$ of principal parts. This sheaf appears naturally in the sequence
\begin {equation}\label {eq:sequence}
0\rightarrow\mathcal O\rightarrow \mathcal M\stackrel P\rightarrow \mathcal P \rightarrow 0,
\end {equation}
where $\mathcal O$ and $\mathcal M$ are the sheaves of holomorphic and meromorphic functions respectively, and $P$ assigns to a meromorphic function its principal parts. We write
\begin {equation}\label {eq:principal}
P_\Theta = P \left(\int \Theta\right),\quad P_\Psi= P\left(\int\Psi\right).
\end {equation}
Taking the \v{C}ech cohomology of  \eqref {eq:sequence} on the spectral curve $X$, since $\mathcal M$ is a fine sheaf, $H ^ 1 (X,\mathcal M) = 0$ and so
\begin{equation}
\label{equation:principal}
H ^ 1 (X,\mathcal O) \cong \frac {H ^ 0 (X,\mathcal P)} {P (H ^ 0 (X,\mathcal M))}.
\end{equation}
Writing  $D =\lambda^{-1} (0+\infty )$ and denoting by $\mathcal O (D) $ the sheaf of meromorphic functions with poles at most on $ D $, we have an exact sequence
\[
0\rightarrow\mathcal O\rightarrow\mathcal O (D)\rightarrow \left. \mathcal O (D)\right|_D\rightarrow 0,
\]
and we can identify $\mathcal O (D) $ and $\left. \mathcal O(D)\right |_D $ as sub-sheaves of $\mathcal M $ and $\mathcal P $ respectively, with this identification the co-boundary map $\delta $ for this sequence is the same as that for  \eqref {eq:sequence}. 
The vectors $ P_\Theta $ and $P_\Psi $ lie in the 2-dimensional real subspace 
\[
 W =\{P\in H ^ 0 (D,\mathcal O (D))\mid\sigma ^*P = - P,\,\overline {\rho ^*P} = - P\} 
\]
 of the four-dimensional complex vector space $ H ^ 0 (D,\mathcal O (D)) $. Let $\Lambda  $ be the lattice in $ W $ defined by $ P_\Theta $ and $ P_\Psi $ and then we have a natural realisation $ W/\Lambda $ of the domain 2-torus of the harmonic map $ f $. Furthermore, the coboundary map $\delta:  W/\Lambda\rightarrow H ^ 1 (X,\mathcal O) /H ^ 1 (X,\Z) =\mathrm{Jac} (X) $ is linear, as we easily see by giving the explicit realisation of this map which we proceed now to do. We shall give details for the case when $ X $ is un-branched over $\lambda = 0,\infty $ or equivalently when the corresponding map $ f $ is non-conformal. The modifications for the conformal case are clear.

Assume then that $ D $ consists of four distinct points, and fix $ p_0\in\lambda^ {- 1} (0) $. Set 
\[
 q_0 =\sigma (p_0),\,p_\infty =\rho (p_0),\, q_\infty =\sigma (p_\infty )\text { and let }
 c = P_\Theta (p_0).
\]
We have an isomorphism 
\begin {align*}
\R ^ 2 &\cong W \\
 z &\mapsto P ^ {c z}
\end {align*}
where 
\[
P ^ {c z} (p_0) = {cz}\lambda^ {- 1},\, P ^ {cz} (q_0) = - {cz}\lambda^ {- 1},\\ P ^ {cz} (p_\infty) = -\bar {cz}\lambda,\, P ^ {cz} (q_\infty) =\bar {cz}\lambda. 
\]
Note that  (2) and  (4) 
 of Theorem~\ref {theorem:Hitchin} give that
\[
P_\Theta = P ^ c,\quad P_\Psi= P ^ {c\tau}.
\]
The scale factor of $ c $ is included in the isomorphism because the lattice $\Lambda $ has basis vectors $ c $ and $ c\tau $, whereas for the original domain torus we normalised these to $ 1 $ and $\tau $.

Let $U_0$ be an open neighbourhood of $p_0\in\lambda^{-1} (0)$ on which $\lambda$ is a local coordinate, and define $V_0 =\sigma (U_0), U_\infty =\rho(U_0), V_\infty =\sigma\rho (U_0)$. These four sets together with $A = X -\{p_0, q_0, p_\infty, q_\infty\}$ form a Leray cover of $X$. Then recalling $\delta: H ^ 0 (D,\mathcal O (D))\rightarrow H ^ 1 (X,\mathcal O)$ denotes the coboundary of the above sequence, $l_{P ^ {cz}} =\delta (P ^ {cz})$ is defined by the cocycles
\[
\begin{aligned}
\left (l_{P ^ {cz}}\right)_{AU_0} &= {cz}\lambda^ {- 1}, &
\left (l_{P ^ {cz}}\right)_{AV_0} &= - {cz}\lambda^ {- 1},\\
\left (l_{P ^ {cz}}\right)_{AU_\infty} &= -\bar {cz}\lambda &
\text{and }
\left (l_{P ^ {cz}}\right)_{AV_\infty} &= \bar {cz}\lambda.
\end{aligned}
\]
Clearly then $\delta $ is linear.
For each $ z\in\C $ then we define a line bundle $\mathcal E_z $ on $ X $ of degree $ g +1 $ by
\[
\mathcal E_z =\exp (l_{P ^ {c z}}) \otimes \mathcal E_0.
\]
(This is consistent with the construction of spectral data from a harmonic map, where $\mathcal E_z $ is defined to be the eigenline bundle with base point $z $.)

Now using the sequence
\[
0\rightarrow\C\rightarrow\mathcal  O (D)\rightarrow d\mathcal O (D)\rightarrow 0,
\]
the periods of $\Theta $ and $\Psi $ are obtained as the images of these differentials under the co-boundary map, identifying these periods as elements of $ H ^ 1 (X,\C) $.
In \cite [pp 664--5] {Hitchin:90}, Hitchin explains how one can use the \v{C}ech cohomology of a commuting diagram of short exact sequences of sheaves to prove that under the natural injection $ H ^ 1 (X,\C)\rightarrow H ^ 1 (X,\mathcal O) $ the periods of $\Theta $ and $\Psi$ correspond to $ -\delta (P_\Theta) $ and $ -\delta (P_\Psi) $ respectively. An explicit argument for this,  which is also more in keeping with the expository flavour of this article, is included in Theorem~\ref {theorem:periodicity}. For simplicity we restrict ourselves to the generic case where $ X $ is smooth.
 In Theorem ~\ref {theorem:singular} we explain how to extend our arguments to the case of a curve with ordinary double points, which is needed in order to extend the various characterisations to include periodicity condition (P2). 

Assume  then that  $ X $ is smooth, so we may choose a standard basis $A_1,\ldots , A_g, B_1,\ldots , B_g$ for the homology of $X$ such that $\rho _*A_i = - A_i$ and $\rho_*(B_i) \equiv B_i\;\mathrm{mod}\; \langle A_1, \ldots , A_g \rangle $. The class $A_i\in H_1 (X,\Z) $ may be represented by a lift of a curve in $\P ^ 1$ with winding number one  about pairs of branch points $\lambda_i$ and $\rho (\lambda_i)$ of $\lambda: X\rightarrow\P ^ 1$ and zero about the other branch points.  Together with the reality conditions on $\Theta,\Psi$, this gives that 
\[
s_i: =\int_{A_i}\Theta,\, t_i: =\int_{A_i}\Psi\text { are real. }
\]
Let $\omega ^ 1,\ldots,\omega ^ g$ be the basis of the holomorphic differentials on $X$ determined by
\[
\int_{A_i}\omega ^ j =\delta_{ij},
\]
and note that since $\int_{A_i}\overline {\rho ^*\omega ^ j} =\overline {\int_{\rho_*A_i}\omega ^ j } = -\overline {\int_{A_i}\omega ^ j} = - \int_{A_i}\omega ^ j $, these differentials satisfy $\overline {\rho ^*\omega ^ j} = -\omega ^ j$.   Hence
\[
\Theta_0: =\Theta -\sum_{j = 1} ^ g s_j\omega ^ j,\,\Psi_0: =\Psi-\sum_{j = 1} ^ gt_j\omega ^ j
\]
are differentials satisfying $\rho^*\Theta_0 = -\bar\Theta_0,\,\rho ^*\Psi_0 = -\bar\Psi_0$ and the criteria of Theorem~\ref{theorem:Hitchin} and enjoying the additional property that their $A$--periods are zero, or equivalently that all their periods are purely imaginary. Define then $\Pi_\Theta,\Pi_\Psi\in H ^ 0 (X,\mathcal K) ^\vee $ by, for $ \omega =\sum_{j = 1} ^ ga_j\omega ^ j $
\begin {equation}\label {eq:pi}
\Pi_\Theta (\omega) =\sum_{j = 1} ^ g a_j\int_{B_j}\Theta_0,\quad\Pi_\Psi (\omega) =\sum_{j = 1} ^ g a_j\int_{B_j}\Psi_0.
\end {equation}



The following theorem allows us to variously express periodicity condition (P1) in terms of derivatives of the linear family $ l_{P ^ {cz}} $, the principal parts $ P_\Theta $ and $ P_\Psi $ or  Abel-Jacobi maps on $ X $. As we show in Theorem~\ref {theorem:singular} it also enables us to express both periodicity conditions (P1) and (P2) together in these various forms. 
\begin {theorem}
\label{theorem:periodicity} Take smooth spectral data $ (X,\lambda,\Theta,\Psi) $ satisfying (\ref{item:first})--(\ref{item:last}) of Theorem ~\ref {theorem:Hitchin} and such that $\Theta $ and $\Psi $ have purely imaginary periods.
 Write $ z = u +\tau v $ for $ u,v\in\R $.
Then we have the following equalities.
\begin {alignat}{3} 
 -\dfrac {1} {2\pi\sqrt {-1}}\Pi_\Theta & = 
\dfrac {\partial l_{cz}} {\partial u} & =\delta (P_{\log \mu}) & =
\left\{\begin{array}{ll}
 c\left.\dfrac {d\mathcal A_{p_0} } {d\zeta}\right|_{\zeta = 0} -\bar c   {\left.\dfrac {d\mathcal A_{p_\infty } } {d\zeta^ {- 1}}\right|_{\zeta = \infty }}, &\\
\qquad \mbox{when $\lambda $ is branched at $ 0,\infty $ and $\zeta ^ 2 =\lambda $;}&\\
 2c\left. \dfrac {d\mathcal A_{p_0} } {d\lambda}\right|_{\lambda = 0}  -2\bar c {\left. \dfrac {d\mathcal A_{p_\infty } } {d\lambda^ {- 1}}\right|_{\lambda = \infty }}, &\\
\qquad\mbox{ when $\lambda $ is unbranched at $ 0,\infty $;}&\end{array}\right.\label {eq:1st}\\
 -\dfrac {1} {2\pi\sqrt {-1}}\Pi_\Psi & =
 \dfrac {\partial l_{cz}} {\partial v} & =\delta (P_{\log \nu}) & =
\left\{\begin{array}{ll}
 c\tau\left.\dfrac {d\mathcal A_{p_0} } {d\zeta}\right|_{\zeta = 0} -\bar c\bar\tau   {\left.\dfrac {d\mathcal A_{p_\infty } } {d\zeta^ {- 1}}\right|_{\zeta = \infty }}, &\\
\qquad\mbox{ when $\lambda $ is branched at $ 0,\infty $ and $\zeta ^ 2 =\lambda $;}&\\
 2c\tau\left. \dfrac {d\mathcal A_{p_0} } {d\lambda}\right|_{\lambda = 0}  - 2\bar c\bar\tau {\left. \dfrac {d\mathcal A_{p_\infty } } {d\lambda^ {- 1}}\right|_{\lambda = \infty }}, &\\
\qquad\mbox{ when $\lambda $ is unbranched at $ 0,\infty $;}&\end{array}\right.\label {eq:2nd}
\end {alignat}
Here $\mathcal A_{p_0} $ denotes the Abel-Jacobi map with base point $ p_0\in\lambda^ {- 1} (0) $ and $ p_\infty =\rho (p_0) $.

Periodicity condition (P1) is equivalent to requiring that the elements of $ H ^ 1 (X,\mathcal O) $ in equations \eqref {eq:1st} and \eqref {eq:2nd} are integral, that is lie in the lattice $ H ^ 1 (X,  \mathbb Z) $. In particular, by the last equality this is determined by just the spectral curve $ X $ and the projection $\lambda $.

\end{theorem}

\begin{proof}
We shall prove the theorem in the case when $\lambda$ is unbranched at $0$ and $\infty$, the branched case being similar. The proof is broken into the verification of the equalities listed below.
\begin {enumerate}
\item 
 $ {\dfrac {\partial l_{c z}} {\partial u} =\delta (P_\Theta),\,\dfrac {\partial l_{cz}} {\partial v} =\delta (P_{\log \nu}) } $ \\
\vspace{5pt}

This is true essentially by definition, since
\[
\dfrac {\partial l_{c z}} {\partial u} =\delta\left (\frac {\partial P ^ {cz}} {\partial u}\right) =\delta (P_\Theta)
\]
and similarly for the other equality.

\vspace{5pt}
\item {$\dfrac {\partial l_{cz}} {\partial z} = 2c
\left. \dfrac {d\mathcal A_{p_0} } {d\lambda}
\right|_{\lambda = 0},\quad\dfrac {\partial l_{cz}} {\partial\bar z} = -2\bar c {
\left.  \dfrac {d\mathcal A_{p_\infty } } {d\lambda^ {- 1}}\right|_{\lambda = \infty }} $}\\
\vspace{5pt}

As
\[
\frac {\partial} {\partial u} =\frac {\partial} {\partial z} +\frac {\partial} {\partial\bar z},\quad\dfrac {\partial} {\partial v} =\tau\dfrac {\partial} {\partial z} +\bar\tau\dfrac {\partial} {\partial\bar z},
\]
proving this is equivalent to demonstrating the statements in \eqref{eq:1st} and \eqref{eq:2nd} relating derivatives of $ l_{cz} $ to derivatives of Abel-Jacobi maps.

It is not difficult to check that
\[
([P],\omega) =\sum_{q\in X}\operatorname {Res}_q (P (q)\omega)
\]
is well-defined and gives a nondegenerate pairing between $\frac {H ^ 0 (X,\mathcal P)} {P (H ^ 0 (X,\mathcal M))}$ and $H ^ 0 (X,\mathcal K)$. 
By the arguments above,
$
\dfrac {\partial l_{P^{cz}}} {\partial z} $
is given by the equivalence class in $\frac {H ^ 0 (\mathcal P)} {P (H ^ 0 (\mathcal M))}$ of the principal part
\[
P (p_0) = c\lambda^ {- 1},\, P (q_0) = - c\lambda^ {- 1}.
\]
Using the above pairing  to consider $\dfrac {\partial l_{cz}} {\partial z}$  as a linear functional on $H ^ 0 (X,\mathcal K)$,
\[
\begin{aligned}
\dfrac {\partial l_{cz}} {\partial z} (\omega) & = \text{Res}_{p_0}\left (\frac {c\omega} {\lambda}\right) -\text{Res}_{q_0}\left (\frac {c\omega} {\lambda}\right)\\
& = c (\omega (p_0) -\omega (q_0))\\
& = 2c\omega (p_0)\\
& = 2 c\left.\dfrac {d} {d\lambda}\right|_{\lambda = 0} \int_{p_0} ^ {\lambda}\omega\\
& = 2 c\left.\dfrac {d} {d\lambda}\right|_{\lambda = 0}\mathcal A_{p_0},
\end{aligned}
\]
where we are using the fact that on a hyperelliptic curve all holomorphic differentials satisfy $\sigma ^*\omega ^ j = -\omega ^ j$.
Similarly
\[
\dfrac {\partial l_{cz}} {\partial\bar z} (\omega) = -2\bar c \left.\dfrac {d} {d\lambda^ {- 1}}\right|_{\lambda = \infty }\mathcal A_{p_\infty }.
\]

\vspace{5pt}
\item {$\Pi_\Theta = - 2\pi\sqrt {-1}
\dfrac {\partial l_{cz}} {\partial u} $, $\Pi_\Psi = - 2\pi\sqrt {-1}\dfrac {\partial l_{cz}} {\partial v} $}
\vspace{5pt}

Until now $ X $ has been any algebraic curve satisfying the symmetries of Theorem~\ref {theorem:Hitchin}. For simplicity we now assume that $ X $ is smooth. 

We can choose representatives $ A_i, B_i$ for our standard homology basis so that each of the curves emanate from a fixed $ x_0\in X
$. Then $\Delta: = X -\bigcup_{i = 1} ^ g (A_i\cup B_i) $ is simply connected, and we may define an entire function on it by
\[
h ^ j (x): =\int_{x_0} ^ x\omega ^ j.
\]
Then since the values of $ h ^ j $ at corresponding points of $ A_i $ and $ A_i^ {- 1} $ differ by the period of $\omega ^ j $ over $ B_j $ and vice versa, (this is the standard reciprocity argument, \cite{GH:94}) 
\begin{align*}
\int_{B_j}\Theta_0 & =\sum_{i = 1} ^ g\left (\int_{A_i}\omega ^ j\int_{B_i}\Theta -\int _{B_i}\omega ^ j\int_{A_j}\Theta\right)\\
    & =\int_{\partial\Delta} h ^ j\Theta_0\\
    & = 2 \pi \sqrt {-1} \sum_{p\in X}\text {Res}_p (h ^ j\Theta_0)\\
    & = - 4 \pi \sqrt {-1} (a _ j c +\overline{a_j}\bar c)
\end{align*}
where in a neighbourhood of $P_0$,
\[
\begin{aligned}
\Theta _0 & = (- c\lambda ^ {-2} +\text {holomorphic}) d\lambda\\
\omega ^ j & = (a_j +\text {higher order terms}) d\lambda.
\end{aligned}
\]
Hence
\begin {align*}
\Pi_\Theta & = 4\pi\sqrt {-1}\left (\bar c   {
\left.  \dfrac {\partial\mathcal A_{p_\infty } } {\partial\lambda^ {- 1}}\right|_{\lambda = \infty }} -c\left. \dfrac {d\mathcal A_{p_0} } {d\lambda}
\right|_{\lambda = 0}\right).
\end{align*}
Similarly,
\[
\Pi_\Psi = 4\pi \sqrt {-1}\left (- c\tau\left. \dfrac {d\mathcal A_{p_0} } {d\lambda}\right|_{\lambda = 0}  + \bar c\bar\tau {\left. \dfrac {d\mathcal A_{p_\infty } } {d\lambda^ {- 1}}\right|_{\lambda = \infty }}\right).\qedhere
\]
\end {enumerate}

\end{proof}

Recall that periodicity condition (P1), namely that $\Theta $ and $\Psi $ have periods lying in $ 2\pi\sqrt {-1}\Z $,  guaranteed that the spectral data corresponded to a harmonic section of an $ SU (2) $-principal bundle over a 2-torus, to obtain a harmonic map we required also periodicity condition (P2). This can also be expressed in terms of integrality of periods but to do so we must pull $\Theta $ and $\Psi $ back to the singular curve $ \hat X $
defined by
\[
y ^ 2 = (\lambda +1) ^ 2 (\lambda -1) ^ 2 a (\lambda)
\]
Otherwise said, $\hat X$ is the curve obtained from $X$ by identifying the two points $p_1, q_1$ in $\lambda^{-1} (1)$ together to form an ordinary double point, and doing likewise with the two points $p_{-1}, q_{-1}$ in $\lambda^ {- 1} (-1)$.

 A line bundle on $\hat X$ can be thought of as a line bundle on $X$, together with $e ^ {c_j}\in\C ^\times $ giving the identification of the points over $p_j, q_j,$ for $j = 1, -1$. As described above, the eigenline bundles $\mathcal E_ z$ are naturally specified with respect to an open cover consisting  of $X_A = X -\lambda^ {- 1}\{0,\infty \}$ and neighbourhoods of the points in $\lambda^{-1}\{0,\infty \}$. These neighbourhoods are taken sufficiently small so that they do not contain any of the points $p_j, q_j$ and so we obtain a linear flow of line bundles $\hat{\mathcal E}_z $ on $\hat X$ by employing the same transition functions.

Periodicity conditions (P1) and (P2) together are now exactly the requirement that (P1) holds for $ \hat X $. Of the various equalities proven in Theorem~\ref{theorem:periodicity}, only for those involving $\Pi_\Theta,\Pi_\Psi $ did our proof utilise the assumption that $ X $ is smooth. We now explain how to modify our interpretation of the periods of $\Theta $ and $\Psi $ as elements of the dual of the space of regular differentials for the case when our curve has a pair of ordinary double points.  


We begin by supplementing our normalised homology basis $ A_j, B_j $ for $ X $ by additional curves which we push forward under the normalisation map
\begin {align*}
\iota:  X &\rightarrow\hat X\\
 (\lambda, y) &\mapsto (\lambda, (\lambda +1) (\lambda -1) y) 
\end {align*}
 to yield a homology basis for $\hat X $. Recall that we represented our basis by curves emanating from a single point $ x_0\in X $ and so that $\rho_*(A_j) \sim - A_j $, $\rho_*(B_j) \sim B_j\;\mbox{mod}\;\langle A_1,\ldots, A_g \rangle $.
For $ k =\pm  1 $ choose an embedded curve $\gamma_k $ from  $ p_k $ to $ q_k $ not intersecting any $ A_j, B_j $ and such that $\rho_*\gamma_k \sim-\gamma_k $.


The regular differentials on $\hat X $ correspond to holomorphic differentials on $ X $ together with meromorphic differentials whose only singularities are simple poles  at the points $ p_1, q_1 $ or $ p_{-1}, q_{-1} $ or both and satisfying
\[
\residue_{p_k}\omega = -\residue_{q_k}\omega\quad\text { for } k = 1, 2.
\]
We take a normalised basis $ H ^ 0 (X,\mathcal K) $ represented by the differentials $\omega ^ 1,\ldots,\omega ^ g $ satisfying $\int_{A _ j}\omega ^ i =\delta ^ i_j $. Define $\eta ^ k $ for $ k =\pm 1 $ to be the meromorphic differential whose only singularities are simple poles at $ p_k, q_k $ such that $\int_{A_j}\eta ^ k = 0 $ for $ j = 1,\ldots, g $ and with residue $ \frac {1} {2\pi\sqrt {- 1}} $ and $ - \frac {1} {2\pi\sqrt {- 1}}  $ at $ p_k $ and $ q_k $ respectively. Assume that $\Theta_0 $, $\Psi_0 $ are normalised as above, that is that their $ A_j $-periods vanish. Then define $\hat\Pi_\Theta,\hat\Pi_\Psi\in H ^ 0 (\hat X,\mathcal K) ^\vee $ by
\begin{align*}
\hat\Pi_\Theta (\omega ^ j) &=\int_{B_j}\Theta_0,\quad\hat\Pi_{\Theta_0} (\eta ^ k) =\int_{\gamma_k }\Theta_0\text { and }\\
\hat\Pi_\Psi(\omega ^ j) &=\int_{B_j}\Psi_0,\quad\hat\Pi_\Psi(\eta ^ k) =\int_{\gamma_k }\Psi_0\quad\text { for } j = 1,\ldots, g, k =\pm1.
\end{align*}
\begin {theorem}\label {theorem:singular}
Let $ (X,\lambda,\Theta,\Psi)$ be smooth spectral data satisfying (\ref{item:first})--(\ref{item:last}) of Theorem ~\ref {theorem:Hitchin} and such that $\Theta $ and $\Psi $ have purely imaginary periods. Let $\hat X $ be the curve obtained from $X$ by identifying the two points $p_1, q_1$ in $\lambda^{-1} (1)$ together to form an ordinary double point, and doing likewise with the two points $p_{-1}, q_{-1}$ in $\lambda^ {- 1} (-1)$. Then the statement of Theorem ~\ref {theorem:periodicity} holds also for $\hat X $, and periodicity condition (P2) is equivalent to the requirement that this pair of elements of $ H ^ 1 (\hat X,\mathcal O) $ are integral with respect to the lattice $ H_1 (\hat X, 2\pi\sqrt {-1}\Z) $.\end {theorem}

\begin {proof}
As noted above, the equalities not involving $\hat\Pi_\Theta $ and $\hat\Pi_\Psi $ were already established in the proof of Theorem~\ref {theorem:periodicity}. We argue now that these are given by the Abel-Jacobi derivatives stated above for the singular curve $\hat X $. As before we assume that $\lambda $ is not branched over $ 0 $ and $\infty $.

Denote by $\Delta $ the simply connected region formed by cutting $ X $ along the homology basis $ A_1,\ldots, A_g, B_1, \ldots , B_g $ 
 specified above. Then 
exactly as in the proof of Theorem ~\ref {theorem:periodicity}, reciprocity yields that
\[
\hat\Pi_\Theta (\omega ^ j) =\int_{B_j}\Theta = 4\pi\sqrt {-1}\left (\bar c   {
\left.  \dfrac {d\mathcal A_{p_\infty } } {d\lambda^ {- 1}} (\omega ^ j)\right|_{\lambda = \infty }}-c\left. \dfrac {d\mathcal A_{p_0} } {d\lambda} (\omega ^ j)
\right|_{\lambda = 0}\right) 
\]
Furthermore, set $\hat\Delta_k =\Delta\setminus\gamma_k $,  fix $ x_0\in\hat\Delta_k $ and define $ l _k (x) =\int_{x_0} ^ x\eta ^k $. 
Then with $\gamma_k ^ + $ and $\gamma_k ^ - $ denoting either side of the split left by the deletion of $\gamma_k $ as shown in Figure~\ref {figure:cuts},
\begin{figure}[htb]
\centering
\includegraphics
{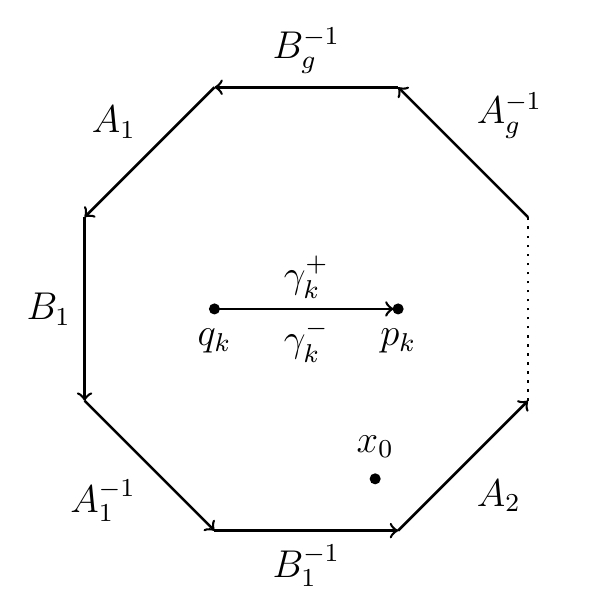}
\caption{Polygon with cut}
\label{figure:cuts}
\end{figure}

\begin {align*}
2\pi\sqrt {-1}\sum_{p\in\hat \Delta_k}\residue_p (l_k\Theta) & =
\int_{\partial\hat\Delta}  l _k\Theta \\
&=\sum_{j = 1} ^ g\left (\int_{A _ j}\eta ^ k\int_{B_j}\Theta -\int_{B_j}\eta ^ k\int_{A_j}\Theta\right) +
\int_{\gamma_k ^ +} l_k\Theta -\int_{\gamma_k ^ -} l_k \Theta\\
& =\int_{\gamma_k ^ +} l_k\Theta -\int_{\gamma_k ^ -} l_k \Theta
\end {align*}
since $\eta ^ k $ and $\Theta $ are both normalised to have vanishing $ A $-periods. But for $p ^ +\in\gamma_k ^ + $ and the corresponding $ p ^ -\in\gamma_k ^ - $,
\[
l_k (p ^ +) - l_k (p ^ -) = 2\pi\sqrt {-1}\residue_{p_k}\eta ^ k = 1,
\]
so
\begin {align*}
\Pi_\Theta (\eta ^ k) =\int_\gamma\Theta & = -2\pi\sqrt {-1}\sum_{p\in\hat\Delta_k}\residue_{p} (l_k\Theta)\\
& = 4\pi\sqrt {-1}\left (\bar c   {
\left.  \dfrac {\partial\mathcal A_{p_\infty } } {\partial\lambda^ {- 1}} (\eta ^ k)\right|_{\lambda = \infty }} -c\left. \dfrac {d\mathcal A_{p_0} } {d\lambda} (\eta ^ k)
\right|_{\lambda = 0}\right).
\end {align*}
Similarly,
\[
\Pi_\Psi = 4\pi \sqrt {-1}\left (- c\tau\left. \dfrac {d\mathcal A_{p_0} } {d\lambda}\right|_{\lambda = 0}  + \bar c\bar\tau {\left. \dfrac {d\mathcal A_{p_\infty } } {d\lambda^ {- 1}}\right|_{\lambda = \infty }}\right).\qedhere
\]
\end{proof}

\section {Spectral curves and their applications} 
\label {section:Lax}

\subsection {Finite-type harmonic maps from a surface to a symmetric space}
We begin by extending the zero-curvature description of harmonic maps into symmetric spaces, as these arise more frequently in geometric applications then do Lie groups. We then introduce the notion of harmonic maps of the plane of {\it finite-type} and explain how a spectral curve construction can be given for these. Certainly not all harmonic maps of the plane are finite-type, but in many situations it has been shown that all doubly-periodic such maps are either totally isotropic (and  given by holomorphic data in terms of a Weierstrass-type representation) or of finite-type and expressible in terms of spectral curve data.
Harmonic maps into Lie groups and symmetric spaces can be studied in terms of one another due to the fact that there is a natural totally geodesic immersion from a symmetric space into the corresponding group, namely the Cartan immersion. As such we could discuss finite-type solutions at either the Lie group or symmetric space level, we choose to do  the latter due to the aforementioned prevalence of symmetric spaces in geometric applications and also because one can always regard a Lie group $ G $ as a symmetric space $ (G\times G)/G $.

A homogeneous space $ G/H $ is a \emph{symmetric space} if there exists an involution $\sigma: G\rightarrow G $ such that 
\[
(G ^\sigma)_0\subset H\subset G ^\sigma
\] where $ G ^\sigma $ denotes the fixed point set of $\sigma $, and $ (G ^\sigma)_0 $ the identity component of $ G ^\sigma $.

Recall that a map $ \iota: N\rightarrow P $ between pseudo-Riemannian manifolds is \emph{totally geodesic} when it sends geodesics to geodesics, or equivalently when
\[
\text {second fundamental form of $ \iota $}= (\nabla (d\iota)) = 0.
\]
 The Cartan map of a symmetric space is given by
\begin {align*}
\iota:\;\; & G/H\rightarrow G\\
& gH\mapsto \sigma (g ) g^{-1}.
\end {align*}
The following result is standard when $ G $ is compact and hence has a bi-invariant Riemannian metric, and is not difficult to extend to the pseudo-Riemannian case (see \cite {CT:11} for details).
\begin {theorem}
Let $ G $ be a semisimple  Lie group with bi-invariant (pseudo)-metric $\langle \cdot,\cdot \rangle $ and $ G/H $ a symmetric space with respect to the involution $\sigma: G\rightarrow G $.
Then $\iota: gH\mapsto \sigma (g)g ^ {- 1} $ is a totally geodesic immersion $ G/H\rightarrow G $.
If $ H = G ^\sigma $, then $\iota$ is additionally an embedding.
\end {theorem}
Since the Cartan immersion $\iota: G/H\rightarrow G $ is totally geodesic, a smooth map $ f: M\rightarrow G/H $ is harmonic if and only if the composition $\tilde{f} = \iota\circ f: M\rightarrow G $ is harmonic. Writing $\Phi = \tilde {f} ^ {- 1} d\tilde {f} $, this is equivalent to
\[
d*\Phi = 0.
\]
It is more useful to phrase this in terms of a lift $ F: U\subset M\rightarrow G $ of $ f $ on a simply connected open set $ U $.
The involution $\sigma:\g\rightarrow\g $ gives a splitting
\[
\g =\mathfrak h\oplus \mathfrak m
\]
into the $ (+1) $- and $ (-1) $-eigenspaces spaces of $\sigma $.
Let  $\varphi : =  F^ {- 1} d  F  $ and  then $\varphi  =\varphi _{\mathfrak{h}} +\varphi _{\mathfrak{m}}$ is the decomposition of $\varphi $
 into the eigenspaces of $\sigma $. 

Since $ \tilde { f} =\sigma ( F)  F^ {- 1} $, we have
\[
\Phi =\tilde f^ {- 1} d\tilde f = F\left (\sigma (F)^ {- 1} d (\sigma (F)) - F^{-1}d F \right) F^ {- 1} = -2\mathrm{Ad}_F (\varphi_{\mathfrak{m}}).
\]

It is now straightforward to verify that a harmonic map $ f $ from a simply connected surface into $ G/H $ is equivalent to an $S ^ 1 $-family of $\g $-valued 1-forms
\[
\varphi _\lambda =\lambda\varphi'_\mathfrak{m} +\varphi_\mathfrak{h} +\lambda^ {- 1}{\varphi''}_\mathfrak{m}
\]
each satisfying the Maurer-Cartan equation
\begin {equation}\label {eq:symmetric}
d\varphi_\lambda +\frac 12 [\varphi_\lambda\wedge\varphi_\lambda ] = 0,
\end {equation}
where $\varphi = F^ {- 1} dF $ for $ F $ a lift of $ f $ into $ G $.

Some solutions to \eqref{eq:symmetric} may be obtained merely by solving a pair of commuting ordinary differential equations on a finite dimensional loop algebra. It is upon these {\em finite-type} solutions that we shall focus and we begin by explaining how harmonic maps into symmetric spaces can be described in this straightforward way.

Let $\Omega\g $ be the loop algebra 
\[
\Omega\g: =\{\xi: S ^ 1\rightarrow\g\mid \xi\text { is smooth.}  \}.
\]
For studying maps into symmetric spaces it is helpful to consider the twisted loop group
\[
\Omega ^\sigma G =\{\gamma: S ^ 1\rightarrow G:\gamma (-\lambda)\} =\sigma (\gamma (\lambda))\}
\]
 and corresponding twisted loop algebra $\Omega ^\sigma\g $. 
The (possibly doubly infinite) Laurent expansion
\[
\xi(\lambda) =\sum_{j} \xi_j \lambda  ^ j,\quad\xi_{\text even}\in\mathfrak h ^\C,\quad\xi_{\text odd}\in\mathfrak m ^\C,\quad\xi_{- j} =\bar\xi_j
\]
allows us to filtrate $\Omega ^\sigma\g ^\C $ by finite-dimensional subspaces
\[
\Omega ^\sigma_d =\{\xi\in\Omega\g\mid \xi_j = 0\text { whenever }\left|j\right| >d\}.
\]

Choose a Cartan subalgebra $\mathfrak t$ of $\g $ such that $\mathfrak t\subset\mathfrak h $
and recall that a non-zero $\alpha \in (\mathfrak{t}^\C)^*$ is a {\em root}  with corresponding {\em root space} $\mathcal{G}^{\alpha}\subset\g ^\C $ 
 if  $[X_1, X_2 ]=\alpha(X_1) X_2 $ for all $X_1\in \mathfrak{t}$ and $ X_2\in \mathcal{G}^{\alpha}$. 
Fix also a set of {\em simple roots}, that is roots $\alpha_1, \ldots, \alpha_N $  such that every root $\alpha $ can be written uniquely as
$$\alpha=\sum_{j=1}^{N} m_j \alpha_j,$$ where the $m_j$ are either all positive integers or all negative integers. 

The roots act also on $\mathfrak k ^\C $ and writing $ \mathfrak n $ for the positive root spaces of $\mathfrak h ^\C $ with respect to this choice of simple roots, we obtain an Iwasawa decomposition
\[
\mathfrak{h}^\C =\mathfrak{n}\oplus\mathfrak{t}^\C\oplus\bar{\mathfrak {n}}.
\]
Define $ r:\mathfrak{h} ^\C\rightarrow\mathfrak{h} ^\C $ by
\[
r (\eta) =\eta_{\bar{\mathfrak{n}}} +\frac 12\eta_{\mathfrak k}
\]
and note that
\[
(\eta dz)_{\mathfrak h} = r (\eta) dz +\overline {r (\eta)} d\bar z.
\]
The following result appears in \cite[Theorem 2.5] {BP:94}; in this context the reader is also referred to  \cite {OU:02}.
\begin {theorem} 
\label {theorem:finite}
Suppose $ G $ is a semisimple Lie group with a bi-invariant pseudo-metric and $ G/H $ is a symmetric space. If $ d $ is a positive odd integer and $\xi:\R ^ 2\rightarrow\Omega ^\sigma_d $ satisfies the Lax pair
\[
\dfrac {\partial\xi} {\partial z} = [\xi,\lambda\xi_d + r (\xi_{d -1})]
\]
then
there exists  $ F:\R ^ 2\rightarrow G $ unique up to left translation, such that the map  $ f:\R ^ 2\rightarrow G/H $ 
framed by $ F $ is harmonic and $\varphi = F^ {- 1} dF $ satisfies
\begin {equation}\label {eq:pair}
\varphi'_{\mathfrak m} =\xi_d,\quad\varphi'_{\mathfrak h} =\xi_{d -1}.
\end {equation}
If $ G $ is compact then global solutions of the Lax pair exist for any choice of initial condition.
\end {theorem}

\newcommand{\p}{\mathfrak{p}}
\begin {proof}
Suppose that $\xi:\R ^ 2\rightarrow\Omega_d $ satisfies  \eqref {eq:pair} and hence also the conjugate equation
\begin {equation}\label {eq:conjugate}
\dfrac {\partial\xi} {\partial\bar z} = - [\xi, \lambda^ {- 1}\xi_{- d} +\overline {r (\xi_{d -1})} ].
\end {equation}
Then
\[
d\xi = [\xi, \phi_\lambda ],
\]
where $\varphi_\lambda = (\lambda\xi_d + r (\xi_{d -1})) d z + (\lambda^ {- 1}\xi_{- d} +\overline {r (\xi_{d -1})}) d\bar z$.
Writing 
\begin {align*}
\varphi & = ( \xi_d + r (\xi_{d-1})) dz + (\xi_{- d} +\overline {r (\xi_{d -1})}) d\bar z\\
& = (\varphi '_{\mathfrak p} +\varphi'_{\mathfrak k}) + (\varphi ''_{\mathfrak p} +\varphi''_{\mathfrak k}),
\end {align*}
 equations \eqref {eq:pair} and \eqref {eq:conjugate} yield
\begin {align}
d\varphi'_\p + [\varphi _{\mathfrak {k}}\wedge\varphi '_\p] & = 0\label {eq:MCg1}\\
d\varphi_{\mathfrak {k}} + \frac 12 [\varphi _{\mathfrak {k}}\wedge\varphi_{\mathfrak {k}}] + [\varphi '_\p\wedge\varphi''_\p] & = 0\nonumber\\
d\varphi''_\p + [\varphi _{\mathfrak {k}}\wedge\varphi ''_\p] & = 0.\nonumber
\end {align}
which are the various components of the Maurer-Cartan equation  \eqref {eq:symmetric}.

Now suppose that $ G $ is compact and note that since the coefficient of $\lambda ^ {d +1} $ on the right hand side of 
\[
Z (\xi) = \frac 12 ( X  (\xi) - i Y  (\xi)) = \lambda\xi_d + r (\xi_{d -1})
\]
 vanishes, this equation defines defines vector fields $ X $, $ Y $ and $Z $  on $\Omega_d $. Taking $\bar Z $ to be the vector field conjugate to $ Z $, then $ X, Y $ commute if and only if  $ [  Z,\bar Z] =  0 $. This follows from a straightforward but tedious computation, using the Jacobi identity as well as the inclusions $ [\mathfrak t ^\C, \mathfrak n]\subset\mathfrak n ,   [\bar{\mathfrak n}, \bar{\mathfrak n}]\subset\bar{\mathfrak n} $, $ [\mathfrak n,\mathfrak n]\subset\mathfrak n $ which one sees for example by taking a Chevalley basis for $\mathfrak k ^\C $.

The flows of the vector fields $ X, Y $ are given by Lax equations $\frac {d\xi} {dx} = [\xi,  \hat X (\xi) ]$ and $\frac {d\xi} {d y} = [\xi, \hat Y (\xi) ]$ for $\hat X,\hat Y :\Omega_d\rightarrow\Omega_d $ and the $  L ^ 2 $ inner product on $\Omega\g $ is ad-invariant. Thus
\begin {align*}
\frac {d} {d x}\langle\xi,\xi\rangle & = 2\langle [\xi,\hat X (\xi) ],\xi\rangle \\
& = -2\langle [\xi,\xi ],\hat X (\xi)\rangle = 0
\end {align*}
and similarly for $ Y $, so the flows evolve on spheres and hence are complete whenever $ G $ is compact.  Therefore this pair of complete commuting vector fields defines an action of $\R ^ 2 $ on $\Omega_d$ via
\[
(x, y)\cdot \xi = X ^ {x}_1\circ X ^ {y}_2 (\xi).
\]
We see that  for any $\psi_0\in\Omega_d$, we may define $\xi:\R ^ 2\rightarrow\Omega_d$ by
\[
\xi (x, y) = (x, y)\cdot\xi_0
\]
and the $\xi $ so defined satisfies \eqref{eq:pair}.\end {proof}


Given a harmonic map $ f:\R ^ 2\rightarrow G/H $, we may define a family of flat connections $\phi_\lambda $ as above and then a solution $\xi: S ^1\rightarrow\Omega_d $ to the Lax pair
\[
\frac {\partial\xi} {\partial z} = [\xi,\phi_\lambda']
\]
is called a  {\em polynomial Killing field} and if furthermore $ \xi_d + r (\xi_{d -1}) =\phi_{\bar z} $ then the polynomial Killing field is said to be {\em adapted}.  As Theorem~\ref{theorem:finite} demonstrates, the harmonic map $ f $ can be recovered from an adapted polynomial killing field. 
\begin {definition}
Harmonic maps $ f:\R ^ 2\rightarrow G/H $ arising from the above construction, or equivalently those possessing an adapted polynomial Killing field are said to be of \emph{finite-type}.
\end {definition}
Harmonic maps into groups can be analysed analogously or by considering the group as a symmetric space.
Harmonic maps of finite type can thus be constructed by remarkably more simple means than general harmonic maps. The obvious question of course is how special are these maps? Certainly by no means all harmonic maps of the plane are of finite type but if we restrict our attention to maps which are periodic with respect to a rank-two lattice $\Lambda\subset\R ^ 2 $ then the compactness of the domain $\R ^ 2/\Lambda $ makes it reasonable to ask whether all such doubly periodic maps are of finite type. The first finite-type results were contained in the work of Hitchin on harmonic maps of a 2-torus $\R ^ 2/\Lambda $ to the 3-sphere \cite {Hitchin:90} and in the study by Pinkall and Sterling \cite {PS:89} of constant mean curvature immersions of the plane in Euclidean-3 space. As described in the previous section, Hitchin gave a complete characterisation of harmonic $\R ^ 2/\Lambda\rightarrow S ^ 3\cong SU (2) $ in terms of a spectral curve. He showed that with the exception of conformal harmonic maps into a totally geodesic $ S ^ 2\subset S ^ 3 $ (the so-called totally isotropic maps), all harmonic maps $\R ^ 2/\Lambda\rightarrow  S ^ 3 $ are of finite type (the totally isotropic  are dealt with by separate means, \cite {Calabi:67}). His approach used the holonomy of the family of flat connections rather than polynomial Killing fields and this result came down to an application of the fact that an elliptic operator on a compact domain (the 2-torus) has but a finite dimensional kernel. 


The approach by Pinkall and Sterling was the first in a series of papers \cite {FPPS:92,BPW:95, BFPP:93} showing a large classes of harmonic maps of tori are of finite type. Indeed a major advantage of the polynomial killing field approach is that it has been more amenable to proving finite-type results. In particular this approach yielded the following quite general two theorems.
\begin {theorem}[\cite {BFPP:93}]
Let $ f:T ^ 2 =\R ^ 2/\Lambda\rightarrow G $ be a semi-simple adapted harmonic map into a compact semi-simple Lie group. Then $ f $ is of finite type.
\end {theorem}
From this point of view the importance of the double-periodicity condition comes from the fact that the 1-form $ 4i\xi_d = f ^*\omega\left ( \frac {\partial} {\partial z}\right) $ is holomorphic, since the harmonic map equation may be expressed as the condition that
\[
f^{-1}\nabla_{\frac {\partial} {\partial\bar z}} ^ Gf_*\frac {\partial} {\partial z} = 0.
\]
On a genus-one surface the only holomorphic differentials are constant.

\begin {theorem}[\cite {BFPP:93}]
 Suppose $ G $ is compact and the symmetric space $ G/H $ has rank one. A harmonic map $ f:\R ^ 2/\Lambda\rightarrow G/H $ is of finite type if and only if it is non-conformal.
 \end {theorem}
The rank is the maximum dimension of a subspace of the tangent space (to any point) on which the sectional curvature is identically zero. Rank one symmetric spaces include spheres and projective spaces.

Of course the most geometrically interesting harmonic maps into symmetric spaces are the conformal ones, and when the target is a sphere or complete projective space, Burstall  \cite {Burstall:95} showed that all but the totally isotropic harmonic maps of 2-tori have lifts into an appropriate flag manifold which are of finite type (this involves expanding the notion somewhat to primitive maps into $ k $-symmetric spaces). The compactness assumption also excludes a number of geometrically interesting situations, but recently Turner and the author have shown \cite {CT:11, CT:12} that maps of 2-tori into $G/T$ possessing a Toda frame are necessarily of finite-type, for $ G $ a simple Lie group with bi-invariant pseudo-metric and $ T $ a Cartan subgroup. In particular then superconformal harmonic maps of 2-tori to de-Sitter spheres $S ^ {2n}_1 $ whose harmonic sequence is everywhere defined lift to maps of finite type. Since Willmore surfaces in $ S ^ 3 $ without umbilic points are characterised by the property that their conformal Gauss maps, which take values in $ S ^ 4_1 $, are harmonic this result yields a simple proof that Willmore tori without umbilic points are all of finite-type. Using the multiplier curve described below, the Willmore result is proven in \cite {Schmidt:02, Bohle:08} without the umbilic assumption.

\subsection {A comparison of spectral curve constructions}

As with Hitchin's holonomy construction for harmonic 2-tori in $ S ^ 3 $, one would like to use polynomial Killing fields to build spectral curves for other harmonic maps. The reason to not necessarily use holonomy directly is simply that beyond the case $ G = SU (2) $ there has not been success in proving in this way that the resulting spectral curve is actually algebraic, that is has finite genus. A map being of finite-type means exactly that the resulting spectral curve has finite genus.  The polynomial Killing fields are solutions to a Lax pair, and there is a long and rich history of spectral curve constructions in the study of solutions to Lax pair equations (see for example \cite {AM:80:1, AM:80:2}). However when studying classical examples such as the Toda lattice, or geodesics on an ellipsoid, or more modern examples such as Higgs bundles, one is concerned with a single solution to the Lax equation. The natural approach is then to take the characteristic polynomial of this solution. When considering harmonic maps, one is presented with an entire algebra of solutions. These solutions for example have differing degrees, and hence the characteristic polynomials clearly yield algebraic curves of different genus. The genus has geometric meaning, since the dimension of the space on which one can choose the eigenline bundle determines the dimension of the family in which the harmonic map lies. The eigenline bundle can usually be chosen from an appropriate Prym variety or perhaps a Prym-Tjurin subvariety of the Jacobian and the dimension of this variety can be computed in terms of the spectral genus. 

In   \cite {McIntosh:95, McIntosh:96}, McIntosh considers the entire algebra of polynomial Killing fields for harmonic 2-tori in $\CP ^ n $ by taking the spectrum of a maximal abelian subalgebra, then proving that the resulting curve is independent of the choice of maximal abelian subalgebra. Another construction which is somewhat simpler is to take the ``eigenline'' curve of (a maximal abelian subalgebra) of the polynomial Killing fields, as in  \cite {FPPS:92}. However McIntosh and Romon have given an example in which the spectral curve obtained by this construction is different to that obtained by using the spectrum \cite {MR:10} (and it is possible to reconstruct the map from the latter but not the former)
and so some caution is required with the eigenline approach.


The spectral curve clearly offers a powerful tool in studying the moduli space of harmonic maps and as such has been instrumental in a number of recent approaches on various geometric conjectures. In   \cite {KSS:10}, the authors present a detailed analysis of the moduli space of  equivariant constant mean curvature tori in $ S ^ 3 $ and show that the spectral curve of any embedded equivariant minimal tori of positive (arithmetic) spectral genus can be deformed through a family of spectral curves of constant mean curvature tori to a curve of geometric genus zero and arithmetic genus one,  which is known not to be the spectral curve of any constant mean curvature torus. This contradiction implies that any embedded equivariant minimal torus in $ S ^ 3 $ must in fact have (arithmetic) spectral genus zero and hence be the Clifford torus,  proving the equivariant case of Lawson's conjecture that the only embedded minimal torus in $ S ^ 3 $ is the Clifford torus. A similar approach was employed in   the preprint \cite{KS:08} to yield the Pinkall-Sterling conjecture that the only embedded constant mean curvature tori in $ S ^ 3 $ are those of revolution, a fortiori as yielding the full Lawson conjecture. An analogous approach to the Willmore conjecture shows promise, with partial results  established in \cite {Schmidt:02} (in this context we mention also the recent announcement of a proof of the Willmore conjecture using rather different methods \cite{MN:12}.

More recently, another spectral curve construction has come into vogue, namely the multiplier, or Fermi curve  \cite {Taimanov:97, Schmidt:02, BLPP:07}. In a sense this has its roots in the holonomy construction, but it is more general in that it does not rely upon the existence of a family of flat connections and so applies to maps which are not necessarily harmonic. Instead one considers maps $ f $ from a Riemann surface $\Sigma $ into $ S ^ 4$ which are merely conformal or equivalently, (quaternionic) holomorphic as maps into $\HP ^ 1\cong S ^ 4 $. As with holomorphic maps into complex projective spaces, such $ f $ correspond to quaternionic line sub-bundles of the trivial rank two quaternionic vector bundle $ V $ on $\Sigma $. Geometrically, the multiplier curve encodes a subspace of the space of Darboux transforms of the original map. These are a natural generalisation of classical Darboux transforms, where two surfaces in $\R ^ 3 $ are classical Darboux transforms of one another if they share a common sphere congruence. More generally, given a surface $\Sigma $ and conformal immersion  $ f: \Sigma\rightarrow S ^ 4 $, a {\em Darboux transform} of $ f $ is a conformal map $\hat f: \Sigma\rightarrow S ^ 4 $ such that
for each $ p\in \Sigma $, $f (p)\neq \hat f (p) $, and there is a smooth oriented sphere congruence $ S: M\rightarrow\{\text {oriented round 2-spheres in } S ^ 4\} $ such that $ S  $ left-envelopes $\hat f $ and $ S  $ both left- and right-envelopes $ f $.

To say that $ S $ left-envelopes $\hat f $ means that for all $p\in M,\,\hat f (p)\in S (p) $, and the oriented great circles in $ S ^ 3 $ corresponding to the tangent planes of $\hat f (M) $ and $ S (p) $ at $\hat f (p) $ differ by left translation in $ S ^ 3\cong S U (2) $. Right-enveloping is defined analogously. Alternatively, considering the oriented Grassmannian of 2-planes in $\R ^ 4 $ as $ S ^ 2\times S ^ 2 $, one can think of having a pair of Gauss maps given by the left and right normals and we require that the left normal of $\hat f $ matches that of $ S $ whilst both the left and right normals of $ f $ match those of $ S $. Darboux transforms of $ f $ are exactly the maps defined away from isolated points by holomorphic sections of the pull-back $\widetilde {V/L} $ of the quotient bundle $ V/L $ to the universal cover of $\Sigma $. The multiplier spectral curve is then the the space of holonomies realised by holomorphic sections of $\widetilde {V/L} $.

For the fundamental case of constant mean curvature tori in $\R ^ 3 $, we have in this quaternionic line bundle also a family of flat connections, gauge equivalent to the family of flat connections described in section~\ref {section:content}. The holomorphic structure on $\widetilde {V/L} $ is precisely the $ (0, 1) $ part of these connections (which is independent of the spectral parameter $\lambda $). Thus sections which are actually parallel with respect to some connection $\nabla ^\lambda $ are in particular holomorphic sections, and the corresponding Darboux transforms are termed $\lambda $-Darboux transforms. These maps are included in Hitchin's study of harmonic maps into $ S ^ 3 $, as the Gauss map of the constant mean curvature surface is harmonic and vice versa. In terms of Hitchin's spectral curve, the fact that the map corresponds to a constant mean curvature torus means precisely that the curve is unbranched over $ 0 $ and $\infty $ and it possesses a holomorphic involution covering $\lambda \mapsto -\lambda $. If we term the quotient of Hitchin's curve by this involution the {\it eigenline spectral curve}, we have the following  \cite {CLP:09}.
%
%
%
%
\begin {theorem} [  \cite {CLP:09}]
The eigenline and multiplier curves of a constant mean curvature torus in $\R ^ 3 $ are not birational, however they have the same normalisation. The multiplier curve is always singular whereas the eigenline curve is generically smooth. \end{theorem}
This point of view is a particularly natural way of recovering the original constant mean curvature immersion.
\begin {theorem}[  \cite {CLP:09}]
The original constant mean curvature immersion $ f: T ^ 2\rightarrow\R ^ 3 $ is given by the limit of the $\lambda $-Darboux transforms as $\lambda $ tends toward $ 0 $ or $\infty $. 
\end {theorem}

Clearly it would be highly desirable to be able to extend spectral curve methods to harmonic maps of surfaces of genus higher than one. The multiplier curve offers an approach here, as by considering only the holomorphic structure (so ``half" the connection $\nabla_\lambda $), one is not so tightly tied to the assumption that the fundamental group must be abelian but can rather consider abelian representations of the holonomy of the more general holomorphic sections.

\bibliographystyle{alpha}

\def\cprime{$'$}


\end {document}